\documentclass[12pt,oneside]{amsart}
\usepackage{amsmath}
\usepackage{amsthm}
\usepackage{amsfonts}
\usepackage{amssymb}
\usepackage{enumerate}
\newtheorem{theorem}{Theorem}

\newtheorem{proposition}[theorem]{Proposition}
\newtheorem{corollary}[theorem]{Corollary}

\theoremstyle{definition}

\newtheorem{example}[theorem]{Example}

\theoremstyle{remark}
\newtheorem{remark}[theorem]{Remark}
\newcommand{\DD}{{\mathbb D}}

\DeclareMathOperator{\Aut}{Aut} 
 
 \DeclareMathOperator{\re}{Re}

\DeclareMathOperator{\im}{Im}

\DeclareMathOperator{\convexhull}{conv}

\renewcommand{\phi}{\varphi}


\begin{document}

\title[Complex geodesics in tube domains and harmonic mappings]{Complex geodesics in tube domains and their role in the study of harmonic mappings in the disc}


\address{Institute of Mathematics, Faculty of Mathematics and Computer Science, Jagiellonian
University,  \L ojasiewicza 6, 30-348 Krak\'ow, Poland}

\author{W\l odzimierz Zwonek}\email{wlodzimierz.zwonek@uj.edu.pl}


\begin{abstract} We continue the research on the structure of complex geodesics in tube domains over (bounded) convex bases. In some special cases a more explicit form of the geodesics than the existing ones are provided. As one of the consequences of our study an effective formula for the Kobayashi-Royden metric in the tube domain $T_{\mathbb B_n}$ at the origin is given. The results on the Kobayashi-Royden metric in a natural way provide versions of the Schwarz Lemma for harmonic mappings.  We also present a result on harmonic mappings defined on the disc that may be seen as a generalisation of the Rad\'o-Kneser-Choquet Theorem for a class of harmonic bivalent mappings that lets understand better the geometry of complex geodesics in tube domains.
\end{abstract}
\maketitle

\section{Introduction} In a recent paper of F. Forstneric and D. Kalaj (\cite{For-Kal 2021}) a close relation between the theory of harmonic mappings defined on the disc into higher dimensional domains and the theory of complex geodesics in tube domains has been extensively used.  The relation is based on the classical property allowing the harmonic function to produce a holomorphic function with the real part equal to the given harmonic one. Consequently, the Lempert theory in the class of tube domains turned out to play a special role in the theory of minimal surfaces and the Schwarz Lemma for harmonic mappings. On the other hand a recent study of S. Zaj\c ac (\cite{Zaj 2015a}, \cite{Zaj 2015b}) who developed methods of Lempert for the convex tube domains made it possible to give a nice geometric description of complex geodesics in that class of domains which subsequently led the authors in \cite{Pfl-Zwo 2018} to determine a simple formula for the behaviour of the real part of the geodesics on the boundary. In this paper we continue the research from \cite{Pfl-Zwo 2018} which lets us understand better how complex geodesics in tube domains look like. For instance the study of the special case of $n=2$ provides a number of special types of complex geodesics which are determined by   harmonic mappings (being the real parts of complex geodesics) between the unit disc (denoted by $\mathbb D$) and the basis (denoted by $\Omega$) of the tube domain $T_{\Omega}$. Roughy speaking these harmonic mappings belong to one of the classes: 
\begin{itemize}
\item homeomorpisms of $\mathbb D$ onto the whole domain $\Omega$,\\
\item homeomorphisms of $\mathbb D$ onto one of two parts of $\Omega$ that are produced by division of $\Omega$ by antipodal lines,\\
\item two-to-one mapping of $\mathbb D$ onto a (non-convex) part of the domain $\Omega$,\\
\item harmonic mapping of a special form onto the segment determined by antipodal points.
\end{itemize}
The above characterization is given in Section~\ref{section:two-dimensional} (Proposition~\ref{prop:geodesics-two-dimensional}).
The better understanding of the case when the real part of the complex geodesic is a two-to-one (or bivalent) mapping from the list above follows from a result that could be seen as a generalization of the celebrated Rad\'o-Kneser-Choquet Theorem and is stated as Theorem~\ref{thm:harmonic-bivalent}. We find Theorem~\ref{thm:harmonic-bivalent} to be interesting in its own and we start the presentation of the proofs with that result in Section~\ref{section:rado-kneser-choquet}. We also cannot exclude that similar charaterizations of $p$-valent harmonic mappings may lead to generalizations of Theorem~\ref{thm:harmonic-bivalent}.

The calculation of the the Kobayashi-Royden metric in the tube domain $T_{\mathbb B_2}$ at the origin reduces to a kind of the infinitesimal version of the Schwarz Lemma for harmonic self-mappings of the unit disc. The problem of finding the effective formula for the Kobayashi-Royden metric $\kappa_{T_{\mathbb B_2}}(0;\cdot)$ was formulated explicitly in \cite{For-Kal 2021} and was earlier discussed in \cite{Bre-Ort-Seip 2021} and \cite{Kov-Yang 2020}. We present its solution probably in the simplest possible way by providing a parametrization of the boundary of the indicatrix and consequently we get a formula for $\kappa_{T_{\mathbb B_2}}(0;\cdot)$ which involves the inverses of elliptic integrals (that appear in the parametrization of the boundary of the indicatrix) as shown in Theorem~\ref{thm:formula-kobayashi-ball}. 

\subsection{Basic definitions, notations and results on Kobayashi-Royden metric} The (holomorphic) Schwarz Lemma when generalized to higher dimensions leads to a variety of holomorphically invariant functions and metrics; for example the Kobayashi-Royden pseudometric. As we deal in the paper only with the infinitesimal version of the Schwarz Lemma we restrict below to the introduction of that one invariant metric. 

Let us recall that for the domain $D\subset\mathbb C^n$ {\it the Kobayashi-Royden pseudometric} is given as:
\begin{equation}\label{eq:definition-kobayashi-royden}
\kappa_D(z;X):=\inf\{t>0: \exists f:\mathbb D\to D, f(0)=z, tf^{\prime}(0)=X\},
\end{equation}
$z\in D$, $X\in\mathbb C^n$.

For the domain $D$, $p\in D$ we denote {\it the (closed) indicatrix at $p$} by
\begin{equation}
I_D(p):=\{X\in\mathbb C^n:\kappa_D(p;X)\leq 1\}.
\end{equation}
Recall that $\kappa_{\mathbb D}(z;X)=\frac{|X|}{1-|z|^2}$, $z\in\mathbb D$, $X\in\mathbb C$. We also get that
any holomorphic $F:D_1\to D_2$ satisfies the inequality $\kappa_{D_2}(F(z);F^{\prime}(z)(X))\leq\kappa_{D_1}(z;X)$, $z\in D_1$, $X\in\mathbb C^n$ which may be seen as a version of the Schwarz Lemma for holomorphic mappings.

The case when the infimum in the definition (\ref{eq:definition-kobayashi-royden}) is attained lets us define the notion of complex geodesics and its left inverse for convex bounded domains. More precisely, a consequence of the Lempert Theorem (see \cite{Lem 1981}) is that for a bounded convex domain $D\subset\mathbb C^n$ and for any $z\in D$, $X\in\mathbb C^n$ we find holomorphic $f:\mathbb D\to D$ and $F:D\to\mathbb D$ such that $f(0)=z$, $\kappa_D(z;X)f^{\prime}(0)=X$ and $F\circ f$ is a holomorphic automorphism of $\mathbb D$ (we may equivalently require $F\circ f$ to be the identity). The mapping $f$ is called then {\it a complex geodesic} and $F$ its {\it left inverse}. Recall that in that situation $I_D(p)$ is convex. We recall that the results just mentioned extend to the domain $D$ being the tube domain over the basis being the bounded convex domain. In other words in the class of tube domains over bounded convex domains complex geodesics passing through a given point in arbitrary direction always exist. In this context it is worth to mention a recent result from \cite{Zim 2020} which shows that tube domains with strongly convex bounded convex bases  are never biholomorphic to convex domains. 

A good reference for the properties of holomorphically invariant functions, including the Kobayashi-Royden pseudometric, is a monograph \cite{Jar-Pfl 2013}.

\section{Some extension of the Rad\'o-Kneser-Choquet Theorem}\label{section:rado-kneser-choquet} As already mentioned we begin with presenting a result being a variation of the Rad\'o-Kneser-Choquet Theorem that is interesting in its own but it is motivated by the fact that we need it to complete the proof of Proposition~\ref{prop:geodesics-two-dimensional}. The standard form of the Rad\'o-Kneser-Choquet Theorem is the following one (it is taken from \cite{Dur 2004}). In the formulations below the harmonic extension of a continuous function $f:\mathbb T\to\mathbb C$ to $\overline{\mathbb D}$ is denoted by the same letter $f$  ($\mathbb T:=\partial\mathbb D$).

\begin{theorem}[Rad\'o-Choquet-Kneser Theorem] Let $f:\partial\mathbb D\to\Gamma$ be a homeomeorphism, where the boundary of the domain $D\subset\mathbb C$ is a Jordan curve $\Gamma$. If $D\subset \mathbb C$ is convex then the harmonic extension of $f$ is onto $D$ and univalent. 

In the case $D$ is not convex the theorem remains valid if we assume additionally that the extension of $f$ satisfies the inclusion $f(\mathbb D)\subset D$.
\end{theorem}

Below we present a result that could be described as a refinement of the Rad\'o-Chocquet-Kneser Theorem with the homeomorphisms on the boundary replaced by bivalent functions onto a proper subset of the boundary of the convex domain. We could not find the result in the literature so we present the proof that  follows the ideas of Kneser as presented in the book \cite{Dur 2004}, Section 3.1.  

\begin{theorem}\label{thm:harmonic-bivalent} Fix two non-trivial closed arcs $\gamma_1,\gamma_2\subset\mathbb T$ such that $\gamma_1\cap\gamma_2$ consists of two points (the arcs have common endpoints) $p$ and $q$ and $\gamma_1\cup\gamma_2=\mathbb T$. Let $\Gamma\subset\partial D$, $\Gamma\neq\partial D$  be a connected closed subset of the boundary of the convex bounded open set $D\subset\mathbb C$ with endpoints $x,y\in\Gamma$ and such that $\Gamma$ is not a segment. Let $f:\mathbb T\to \Gamma$ be a continuous mapping such that $f_{|\gamma_j}:\gamma_j\to\Gamma$ is a homeomorphism, $j=1,2$. Let $l(x,y)$ be the open segment joining $x$ and $y$. 

Then there is a convex open set $D_1\subset \convexhull(\Gamma)$ with $L:=l(x,y)\cap \overline{D_1}\subset\overline{D_1}$ such that $f(\mathbb D)=(\operatorname{int}\convexhull(\Gamma))\setminus D_1$, where $\convexhull(\Gamma)$ denotes the convex hull of $\Gamma$. Moreover, there is an analytic Jordan arc $\beta$ with endpoints $p$ and $q$ and with all other elements lying in $\mathbb D$ and thus dividing $\mathbb D$ into two domains $G_j$ such that $f(\beta)=\partial D_1\setminus L$, $f_{|G_j}:G_j\to f(\mathbb D)\setminus f(\beta)$ is univalent and onto, $j=1,2$ and $f_{|\beta}$ is a homeomorphism between $\beta$ and $f(\beta)$.
 \end{theorem}

\begin{remark} Let us make some remarks. The harmonic extension of $f$ onto $\overline{\mathbb D}$ maps $\overline{\mathbb D}$ into $\convexhull{\Gamma}$. The function $f$ is bivalent between $\mathbb T$ and $\Gamma$ (in the sense that the preimage of any element consists of at most two points). The only two points in $\Gamma$ having exactly one element in the preimage are the points $x,y$ (and the preimages are up to a permutation $p$ and $q$). One of the immediate consequences of the theorem is that $f$ is bivalent, too. Additionally, we may divide the unit disc by the Jordan arc $\beta$ into two subdomains with the common boundary in $\mathbb D$ such that the extension is a one-to-one mapping when restricted to the subdomains and on the common boundary of the subdomains it is a homeomorphism. It also follows from the theorem that the points lying  in $f(\beta)$ (being a part of the boundary of $f(\overline{\mathbb D})$) are precisely the images of points from $\beta$ and are also the degenaracy points of $df$. 

The mapping $\re f$ from the last case of Proposition~\ref{prop:geodesics-two-dimensional} satisfies the assumptions of the above theorem.
\end{remark}

Note that the above theorem shows a difference between the case when the extension of the harmonic mapping is a two-fold winding up of the circle onto a Jordan curve and the form of the bivalent mapping of the circle onto a proper part of the boundary of the bounded convex domain as considered in the theorem. The first case may be expressed on $\mathbb T$ by the formula $f(e^{it})=g(e^{i\varphi(t)})$, where $\varphi:[0,2\pi]\to\mathbb R$ is increasing and $\varphi(2\pi)=\varphi(0)+4\pi$ whereas the second one may be expressed by the formula $f(e^{it})=g(e^{i\psi(t)})$, where $\psi:[t_0,t_0+2\pi]\to\mathbb R$ is such that there is a $t_1\in(t_0,t_0+2\pi)$ such that $\psi_{|[t_0,t_1]}$ is increasing, $\psi_{[t_1,t_0+2\pi]}$ is decreasing, $\psi(t_0)=\psi(t_0+2\pi)$ and $\psi(t_1)<\psi(t_0)+2\pi$.

We draw readers' attention to papers \cite{She-Sma 1985}, \cite{Bsh-Hen-Lyz-Wei 2001}, \cite{Bsh-Lyz 2012} and references there to consult results in both a positive direction (i. e. the finiteness of the preimages of such mappings) and negative results which deliver a number of examples of harmonic mappings having for instance the higher valency in the disc than that in the boundary in the case that $f$ is the winding-up of the circle.

The idea of the proof below comes from the proof of Theorem Rad\'o-Choquet-Kneser as presented in \cite{Dur 2004}, Section 3.1. At some points the methods are almost identical whereas  at some places more subtle methods are needed to work in the situation presented in the theorem.

\begin{proof} 
Note that the harmonicity of the extension of $f$ onto $\mathbb D$ and the form of $f=:u+iv$ on $\mathbb T$ imply that $f(\mathbb D)$ is closed in the interior of $\convexhull{(\Gamma)}$ and $L\cap f(\mathbb D)=\emptyset$.

Let $w_0=f(z_0)$ with $z_0\in\mathbb D$.

Proceeding as in \cite{Dur 2004} Section 3.1 but remembering that $f$ when restricted to the circle $\mathbb T$ is bivalent we get the following property: 

{\bf Property.} For any $(a,b)\in \mathbb R^2\setminus\{(0,0)\}$ the connected component $V$ of the set $\{z\in\mathbb D: au(z)+bv(z)=au(z_0)+bv(z_0)\}$ is either an analytic curve or it is a union of two analytic curves intersecting orthogonally at some $z_1\in\mathbb D$. Moreover, in both cases the analytic curves have no intersection in $\overline{\mathbb D}$ with the exception of the second case and the point $z_1$. The point $z_1$ is the only point $z$ from $V$ such that the derivative $d_zf$ vanishes at $(a,b)$ or the system of equations $au_x(z)+bv_x(z)=au_y(z)+bv_y(z)=0$ is satisfied.


Consider now the case when the differential $d_{z_0}f$ is degenarate, or $u_x(z_0)v_y(z_0)=u_y(z_0)v_x(z_0)$. Take any line $w_0+P$ passing through $w_0$ such that $P$ is orthogonal to the image $d_{z_0}f(\mathbb R^2)$. The line is defined by the formula $au+bv=au(z_0)+bv(z_0)$ where $(a,b)$ is a non-trivial solution of the system of linear equations  $au_x(z_0)+bv_x(z_0)=au_y(z_0)+bv_y(z_0)=0$. 
Then following the reasoning of Kneser as presented in \cite{Dur 2004} and using the fact that $f_{|\mathbb T}:\mathbb T\to\Gamma$ is a two-to-one mapping (with the exception of endpoints) we get that $w_0+P$ intersects $\Gamma$ at two points, both different from $x$ and $y$ and $(w_0+P)\cap\convexhull{\Gamma}\subset f(\overline{\mathbb D})$. Moreover, the connected component of $f^{-1}(w_0+P)$ containing $z_0$ is a union of two analytic curves intersecting at $z_0$ that have no self-intersection in $\overline{\mathbb D}$ and that intersect orthogonally at $z_0$.  

Note that the property just mentioned also implies that if the differential $d_{z_0}f$ would be trivial then all the lines passing through $w_0$ would have to touch $\Gamma$ at two points. This is impossible which means that $d_{z_0}f$ can never vanish.

The two analytic curves $C_1(z_0)$, $C_2(z_0)$ associated with the degenarate point $z_0$ split into four analytic arcs $A_j(z_0)$ starting at $z_0$ in $\mathbb D$, $j=1,2,3,4$. We claim that $f$ restricted to $C_j(z_0)$ is injective. Actually if that were not the case then it would deliver an element $z_1\neq z_0$ on the curve $C_j(z_0)$ such that $d_{z_1}f$ would vanish at the vector tangent to the curve at $z_1$. This however would imply that the curve $C_j(z_0)$ branches at $z_1$ which is impossible.

Assume that $f(z_1)=w_0$ for some $z_1\neq z_0$. Then the injectivity of $f$ on the curves implies that $z_1$ cannot lie on any arc $A_j$ so it must lie in the interior of one of the sectors bounded by $A_j$'s and $\mathbb T$. The maximum principle together with the bivalence of $f$ on $\mathbb T$ would however imply that the values of the interior of the sector all lie either below or all lie above the line $w_0+P$ - but the value $w_0=f(z_1)$ lies on the line $w_0+P$ so the maximum principle applied once more would imply that the image of the whole $f$ would lie in $w_0+P$ -- a contradiction. Consequently, for any $z_0\in\mathbb D$ such that $d_{z_0}f$ is degenerate we get that $f^{-1}(f(z_0))=\{z_0\}$.

Let $D_1$ be the interior of the connected component of $\convexhull{(\Gamma)}\setminus f(\mathbb D)$. Note that the closure of $D_1$ contains a neighborhood of $\operatorname{int}L$ in $\convexhull{\Gamma}$. We show that $D_1$ is convex. Actually, take any point $w_1\in\partial D_1$. It is sufficient to show that $D_1$ has a supporting line at $w_1$. The only non-trivial case is when $w_1\in\operatorname{int}\convexhull{\Gamma}$. So fix such a point.
Then there is a point (exactly one!) $z_1\in\mathbb D$ such that $f(z_1)=w_1$ and the differential $d_{z_1}f$ is degenarate. The earlier reasoning implies the existence of a line $P$ with $(w_1+P)\cap\convexhull{\Gamma}\subset f(\overline{\mathbb D})$ which delivers the existence of a suitable supporting line. 

We already know that the preimages of elements of $\partial D_1\setminus L$ consist of exactly one element. This lets us define $\beta:=f^{-1}(\partial D_1\setminus L)$. We also get in the standard way that $f^{-1}$ is continuous on $f(\beta)$. The convexity of $D_1$ also gives the continuity of the mapping $f^{-1}$ up to the boundary of $f(\beta)$. Then the fact that the endpoints of $\beta$ are endpoints of $\Gamma$ and their images under $f^{-1}$ are the points $u,v$ gives the division of $\mathbb D$ into the claimed subdomains $G_j$.

Now we make use of the general version of the Rad\'o-Choquet-Kneser Theorem for the function $f$ composed with the Riemann mapping between $\mathbb D$ and $G_j$ to finish the proof.

\end{proof}

 

\section{Complex geodesics in tube domains over strongly convex bounded bases}
\subsection{The general case}
The results of S. Zaj\c ac (\cite{Zaj 2015a}, \cite{Zaj 2015b}) on the Lempert theory (see \cite{Lem 1981}) applied to convex tube domains allowed the authors in \cite{Pfl-Zwo 2018} to get a complete and in some sense effective characterization of complex geodesics in sufficiently regular convex tube domains which we recall below. The form we recall is the starting point for our further research that leads to a number of results that present properties of different types of complex geodesics in tube domains with regular bases.

First let us introduce basic notations. For a domain $\Omega\subset\mathbb R^n$ denote {\it the tube domain over the basis $\Omega$} as
\begin{equation}
T_{\Omega}:=\Omega+i\mathbb R^n\subset\mathbb C^n.
\end{equation}
Unless otherwise stated throughout the paper we assume additionally that $\Omega$ has $C^k$ boundary ($k\geq 2$) and is strongly convex, which implies in particular its convexity. 

For $x\in\partial\Omega$ denote $\nu_{\Omega}(x):=\frac{\nabla\rho(x)}{||\nabla\rho(x)||}$, where $\rho$ is the defining function of $\Omega$ near $x$, {\it the unit outer normal vector to $\partial \Omega$ at $x$}. This gives us {\it the Gauss map } ($\mathbb S^{n-1}$ denotes the unit sphere and $\mathbb B_n$ denotes the unit Euclidean ball in $\mathbb R^n$)

\begin{equation}
\Phi:\partial \Omega\owns x\to\nu_{\Omega}(x)\in \mathbb S^{n-1}\subset\mathbb R^n,
\end{equation}
which is $C^{k-1}$-smooth, injective and onto (here we need the strong convexity of $\Omega$, its boundedness and smoothness!). 
Consequently, $\Phi$ is a $C^{k-1}$-diffeomorphism.

Below we see that in this setting we have a complete characterization of complex geodesics in $T_{\Omega}$.

Following the results in \cite{Zaj 2015a}, \cite{Zaj 2015b} one concluded in Section 4.3 in \cite{Pfl-Zwo 2018} (see also a similar description in \cite{Blo 2014}) that complex geodesics $f:\mathbb D\to T_{\Omega}$ are determined by the existence of $a\in\mathbb C^n$, $a\neq 0$ and $b\in\mathbb R^n$ such that:

for all but at most two points $\lambda\in\mathbb T$ (when the denominator equals zero) 
we have
\begin{equation}
\re f(\lambda)=
\Phi^{-1}\left(\frac{2\re(\lambda a)+b}{||2\re(\lambda a)+b||}\right).
\end{equation}
Recall that the Poisson formula allows us to determine the values of $f$ in $\mathbb D$ by the following formula
\begin{equation}\label{eq:poisson-formula}
f(\lambda)=\frac{1}{2\pi}\int_0^{2\pi}\frac{e^{it}+\lambda}{e^{it}-\lambda}\Phi^{-1}\left(\frac{2\re(\lambda a)+b}{||2\re(\lambda a)+b||}\right)dt+i\im f(0),\;\lambda\in\mathbb D.
\end{equation} 
The above formula let us state in \cite{Pfl-Zwo 2018} that the behaviour of the projection onto $\mathbb S^{n-1}$ of the following mapping (generically parametrization of the ellipse) is crucial in the context of complex geodesics:
\begin{equation}
\tilde F:\mathbb T\owns\lambda\to 2\re(\lambda a)+b\in\mathbb R^n.
\end{equation}
Let us denote the projection onto $\mathbb S^{n-1}$ by $\pi$: $\pi(x):=\frac{x}{||x||}$, $x\in\mathbb R^n\setminus\{0\}$. 

The assumption $\re f(0)\in\Omega$ implies that the projection of the image of the above mapping onto $\mathbb S^{n-1}$ is not a singleton.

Denote $F:=\pi\circ \tilde F:\mathbb T\to\mathbb S^{n-1}$ which is well defined for all but at most two elements of $\mathbb T$.
\begin{remark} As we saw above the boundary behavior of real parts of complex geodesics in convex tube domains is well understood. First we project by $\pi$ the ellipses (in singular case the intervals) on the unit sphere. Then we transform it by the mapping $\Phi^{-1}$. This gives us the geometric structure of boundaries of real parts of complex geodesics. Certainly, the mapping $\re f:\mathbb D\to\Omega$ is also harmonic, which by the Cauchy-Riemann equations determines the complex geodesic $f$, up to a purely non-real constant $i \im f(0)$.
\end{remark}

\subsection{Special cases}
\begin{remark} 
Note that the linear subspace spanned by the set 
\begin{equation}
\left \{\frac{\re(\lambda a)+b}{||\re(\lambda a)+b||}:\lambda\in\partial \mathbb D\right\}\subset\mathbb S^{n-1}\subset\mathbb R^n
\end{equation}
is at most three-dimensional. Consequently, in the case of $\Omega=\mathbb B_n$ we may define the at most three-dimensional linear subspace determined by the boundary values of real part of the complex geodesic $f$ (with the associated $a,b$) by $V$. The harmonicity of $\re f$ implies then that $\re f(\mathbb D)\subset V$. So the problem of describing the complex geodesics of $T_{\mathbb B_n}$ reduces to the three-dimensional case $T_{\mathbb B_3}$.  This follows easily by applying suitable automorphisms of $T_{\mathbb B^n}$ determined by the real unitary mapping $U$ which lets us map $V$ into $\mathbb R^3\times\{0\}^{n-3}$. This means in particular that when trying to understand how complex geodesics in $T_{\mathbb B^n}$ look like we may restrict ourselves to the case $n=3$.
\end{remark}

\begin{proposition} Assume that the $C^2$-strongly convex domain $\Omega\subset\mathbb R^n$ is symmetric around zero, i. e. $\Omega=-\Omega$. Then all (real parts of) the complex geodesics in $T_{\Omega}$ passing through $0$ are determined by the ones with $b=0$, i. e. any holomorphic geodesic $f:\mathbb D\to\Omega$, with $f(0)=0$ satisfies
\begin{equation}
\re f(\lambda)=\Phi^{-1}\left(\frac{\re(\lambda a)}{||\re(\lambda a)||}\right),\;\lambda\in\mathbb T,
\end{equation}
where $a\in\mathbb C^n\setminus\{0\}$.
\end{proposition}
\begin{proof}
The symmetry of $\Omega$ implies that the diffeomorphism $\Phi$ satisfies the property $\Phi(-x)=-\Phi(x)$, $x\in\partial\Omega$. Let $f:\mathbb D\to T_{\Omega}$ be a complex geodesic with $f(0)=0$. The uniqueness of complex geodesics passing through $0$ in the fixed direction $f^{\prime}(0)$ (compare the proof of Proposition 10 in \cite{Pfl-Zwo 2018}) implies that $f(-\lambda)=-f(\lambda)$, $\lambda\in\mathbb D$. Let $a,b$ be associated to $f$. Our aim is to show that $b=0$. The symmetry of $f$ implies that 
\begin{equation}\label{eq:antisymmetry}
\frac{-2\re(\lambda a)+b}{||-2\re (\lambda a)+b||}=-\frac{2\re(\lambda a)+b}{||2\re (\lambda a)+b||}, \;\lambda\in\mathbb T.
\end{equation}
That means that for any $|\lambda|=1$ the vectors $-2\re (\lambda a)+b$, $-2\re(\lambda a)-b$ are co-linear. Therefore, $\re a$, $\im a$, $b$ are co-linear. In the case $b\neq 0$ this easily leads to a contradiction with the the property (\ref{eq:antisymmetry}).
 \end{proof}

\begin{remark}
The above result implies that for $\Omega$ symmetric the real parts of complex geodesics passing through $0$ are given at boundary point $\lambda$ by the formula $\Phi^{-1}\left(\frac{\re(a\lambda)}{||\re(a\lambda)||}\right)$. Note that in the case $\Omega=\mathbb B_n$ the (real parts of) complex geodesics passing through the origin lie in the two-dimensional linear subspace. Thus the problem of description of complex geodesics in $T_{\mathbb B_n}$ passing though the origin reduces to the case $n=2$. 
\end{remark}

\subsection{General situation once more} Let us come back to the general situation and let us recall the description of complex geodesics in the tube domains that is presented in \cite{Pfl-Zwo 2018} (Section 4.3).

The situations which we list below reflect the fact that the image of the mapping $\tilde F:\mathbb T\owns\to2\re(a\lambda)+b$ has two possibilities and more subcases. The image of the mapping is either an ellipse (co-planar with the origin or not), which is the case when the vectors $\re a$ and $\im a$ are $\mathbb R$--linearly independent, or a closed line segment (with the origin lying in the segment or not).

\begin{remark}\label{remark:geodesics-different-types}
Note that the case when $\tilde F$ has the image being the line segment, i. e. $\re a$ and $\im a$ are linearly dependent, may split into two subcases. With $b$ being colinear with $\re a$ and $\im a$ and the case when $b, \re a$ or $b,\im a$ are linearly independent. In the first case $0$ must lie in the interior of the segment being the image of $\tilde F$ (otherwise $f$ would be constant).

Below we present all the possibilities we have to study that follow directly from the geometry of the image of $F$ and are in principle listed in Section 4.3 in \cite{Pfl-Zwo 2018} (modulo some obvious corrections):
\begin{itemize}
\item $F$ is an embedding of the circle into a sphere.\\
\item $F$ is a real analytic mapping, the image is a (closed) arc of a great circle that is folded up twice by $F$; more precisely, the mapping $[t_0,t_0+2\pi]\owns t\to F(e^{it})=\gamma(e^{i\rho(t)}$), where $t_0<t_1<t_0+2\pi$ and $\rho:[t_0,t_0+2\pi]\to\mathbb R$ is such that $\rho_{[t_0,t_1]}$ is increasing, $\rho_{[t_1,t_0+2\pi]}$ is decreasing $\rho(t_0)=\rho(t_0+2\pi)$, $\rho(t_1)<\rho(t_0)+2\pi$ and $\gamma:[\rho(t_0),\rho(t_1)]\to\mathbb S^{n-1}$ is an arc of a great circle.\\
\item $F:\mathbb T\setminus\{\lambda_0\}\to\mathbb S^{n-1}$ is a real analytic diffeomorphism onto the image being the big open semicircle such that
\begin{equation}
\lim_{t\to t_0^+}F(e^{it})=-\lim_{t\to t_0^-}F(e^{it}),\; \lambda_0=e^{it_0}.
\end{equation}
\\
\item $F:\mathbb T\setminus\{\lambda_0,\lambda_1\}\to\mathbb S^{n-1}$ is constant on the two connected components (arcs) of $\mathbb T\setminus\{\lambda_0,\lambda_1\}$ and the two values are opposite.
\end{itemize}
By the appropriate choice of $a,b$ all the possibilities listed above do occur.
\end{remark}

To have a better insight into the geometry of complex geodesics in tube domains we would require the notion of {\it antipodal points $x,y\in\partial \Omega$}, which are the points such that the tangent hyperplanes are parallel, or equivalently, the normal vectors satisfy the equality $\nu_{\Omega}(x)=-\nu_{\Omega}(y)$. The antipodal points are the ones that are the images of antipodal points in the sphere by $\Phi^{-1}$ and as we have already seen (and we shall see in the sequel) they play a special role when considering the complex geodesics of one of the last two subcases listed above. Note also that in the case $n=2$ the real line $l(x,y)$ joining antipodal points $x$ and $y$ divides the domain $\Omega$ into two parts. 

We know that the real parts of complex geodesics play a special role in the theory of minimal surfaces as noted in \cite{For-Kal 2021}. The harmonic mappings may have in some cases a simple form and their image will be a segment which represents the last subcase. Below we present its effective formula.

\begin{remark}\label{remark:antipodal-geodesic} Combining the formula (\ref{eq:poisson-formula}) applied to the last case in Remark~\ref{remark:geodesics-different-types} we get that two antipodal points $x,y\in\partial\Omega$ determine (up to a composition with the M\"obius function) the complex geodesic 
\begin{equation}\label{eq:one-dimensional-geodesic}
f(\lambda):=\frac{x+y}{2}+\frac{y-x}{2}\frac{2i}{\pi}\log\frac{1+\lambda}{1-\lambda},\;\lambda\in\mathbb D.
\end{equation}
The above geodesics have the simplest possible geometry - one-dimensional in the sense that their real parts are harmonic mappings onto segments. 

The above fact may also be seen directly as follows. For two points $x,y\in\partial\Omega$ we may construct an analytic disc defined on the strip $H:=\{\lambda\in\mathbb C:-1<\re \lambda<1\}$ as follows
\begin{equation}
\psi:H\owns\lambda\to\frac{x+y}{2}+\frac{y-x}{2}\lambda\in T_{\Omega}.
\end{equation}
Without the loss of generality (transforming $\Omega$ by affine automorphism that induces a holomorphic affine automorphism of $T_{\Omega}$) we may assume that $y=-x=(-1,0,\ldots,0)$ and $\Omega\subset(-1,1)\times\mathbb R^{n-1}$ (here we also use the fact that $x,y$ are antipodal). And then the projection $\operatorname{pr}$ on the first variable in $\mathbb C^n$ gives that $\operatorname{pr}\circ\psi$ is the identity on $H$, which implies that $\psi$ is a complex geodesic modulo a conformal mapping between $\mathbb D$ and $H$ - calculating it we get exactly the formula for $f$ as above.

In the subsequent section we shall discuss the complex geodesics in tube domains over two-dimensional bases in all possible cases. The situation will be more complicated but still the structure could be well understood.

\end{remark}

\subsection{Complex geodesics in tube domains over two dimensional basis}\label{section:two-dimensional}
The case $n=2$ forces the real parts of complex geodesics to map the unit disc into planar domains. As remarked in the last subsection the case the image is one dimensional is well understood. All properties of the next proposition except for the last one follow directly from Remark~\ref{remark:geodesics-different-types}, the fact that the map $\re f$ is a diffeomorphism in two first cases is a consequence of the versions of the Rad\'o-Choquet-Kneser Theorem as formulated in \cite{Dur 2004} Sections 3.1 and 3,2 (with the remark following the formulation of the strong form of the theorem); the last case is a direct consequence of Theorem~\ref{thm:harmonic-bivalent}.

\begin{proposition}\label{prop:geodesics-two-dimensional} Let $\Omega$ be a $C^k$-strongly convex domain in $\mathbb C^2$ and let $f:\mathbb D\to T_{\Omega}$ be a complex geodesic. Then $f$ is of one of the following forms.

\begin{itemize}

\item $\re f:\mathbb D\to \Omega$ is a harmonic mapping such that its extension on the boundary is a homeomorphism (of the special form); 
consequently $\re f$ is a harmonic diffeomorphism.\label{homeomorphism-onto}\\

\item $\re f:\mathbb D\to \Omega$ is such that its extension maps $\mathbb T\setminus\{e^{it_0}\}$ homeomorphically onto the boundary of $\Omega$ lying between two antipodal points $x,y\in\partial \Omega$ and $\lim_{t\to t_0^-}\re f(e^{it})=x$, $\lim_{t\to t_0^+}\re f(e^{it})=y$ . Consequently, $\re f$ is a diffeomorphism onto one part of $\Omega$ divided by the real line joining $x$ and $y$.\\

\item $\re f$ is a harmonic mapping onto the line $l(x,y)$ joining two antipodal points $x,y\in\partial\Omega$. The formula is as in (\ref{eq:one-dimensional-geodesic}).\\
  
\item $\re f$ is a harmonic mapping such that its extension maps $\partial\mathbb D$ onto part of the boundary of the domain lying above (or below) the line $l(x,y)$ joining antipodical points $x,y$. Moreover, $\re f$ maps bivalently $\mathbb D$ onto some proper subdomain of $\Omega$.
 \end{itemize}
\end{proposition}

Actually, an even more precise description of how $\re f(\mathbb D)$ and $\re f$ look like is presented in Theorem~\ref{thm:harmonic-bivalent}.

Note that in the situation as in Proposition~\ref{prop:geodesics-two-dimensional} holomorphic geodesics may join two different points having the same real part only in the third and last case. In the case of $\Omega$ being additionally symmetric around zero the antipodal points are opposite so the last case makes it impossible have the origin in the image of the real part of the geodesics. This implies the following property.

A direct consequence of the above result together with the form of complex geodesics in tube domains over bases symmetric around zero is the following.
 
\begin{corollary}
Let $\Omega\subset\mathbb R^2$ be a $C^2$-strongly convex bounded domain symmetric around zero. Let $f$ be a complex geodesic in $T_{\Omega}$ joining $0$ and $0+iv$ for some $v\in\mathbb R^2$. Then $f$ is (up to a M\"obius mapping) of the form as in Remark~\ref{remark:antipodal-geodesic} (formula (\ref{eq:one-dimensional-geodesic})) with $y=-x$ for some $x\in\partial\Omega$.
\end{corollary}

\section{Geodesics in the tube domain over the ball} In the case of the tube domain over the ball the geodesics passing through the $0$ and such that $f(0)=0$ are nicely described by the Poisson formula that in this case has the following form
\begin{equation}\label{eq:poisson}
f(\lambda)=\frac{1}{2\pi}\int_0^{2\pi}\frac{e^{it}+\lambda}{e^{it}-\lambda}\frac{\re(e^{it}a)}{||\re(e^{it}a)||}dt,
\end{equation}
where $a\in\mathbb C^n$. Consequently, we have
\begin{equation}\label{eq:derivative}
f^{\prime}(0)=\frac{1}{\pi}\int_0^{2\pi}e^{-it}\frac{\re(e^{it}a)dt}{||\re(e^{it}a)||},
\end{equation}
which gives the following
\begin{equation}
f^{\prime}(0)=\frac{1}{\pi}\int_0^{2\pi}e^{-it}\frac{\cos t\re a-\sin t\im a}{||\cos t\re a-\sin t\im a||}dt.
\end{equation}
Then after elementary transformations we get
\begin{multline}
f^{\prime}(0)=\\
\frac{1}{\pi}\int_0^{2\pi}\frac{\cos^2 t\re a-\cos t\sin t\im a+i(\sin^2t\im a-\sin t\cos t\re a)}{||\cos t\re a-\sin t\im a||}dt=\\
\frac{1}{\pi}\int_0^{2\pi}\frac{\cos^2 t\re a-\cos t\sin t\im a+i(\sin^2t\im a-\sin t\cos t\re a)}{\sqrt{\cos^2t||\re a||^2+\sin^2 t||\im a||^2-2\langle \re a,\im a\rangle\cos t\sin t}}dt.
\end{multline}


\begin{remark}\label{remark:orthogonality} Note that we may make some simplification in the study of complex geodesics in the tube domain $T_{\mathbb B_n}$ passing through $0$ to the ones with the tangent vector $X+iY\subset\mathbb R^n+i\mathbb R^n$ such that $\langle X,Y\rangle=0$ ($\langle\cdot,\cdot\rangle$ denotes the standard scalar product in $\mathbb R^n$). In fact this would follow from the following fact. The complex geodesic for $(0;X+iY)$ also produces a complex geodesic for $(0;\omega (X+iY))$ for any $|\omega|=1$. In other words the reduction will be possible if for some $\omega=e^{i\theta}$ we shall get that real and imaginary parts of $\omega(X+iY)$ are orthogonal. But
\begin{multline}
\omega(X+iY)=
(\cos \theta X-\sin \theta Y)+i(\cos \theta Y+\sin \theta X)=:\widetilde{X}+i\widetilde{Y}.
\end{multline}
We need to find $\theta\in\mathbb R$ such that $\langle \widetilde{X},\widetilde{Y}\rangle=0$ which gives the condition
\begin{equation}\label{eq:orthogonality}
\cos(2\theta)\langle X,Y\rangle+\frac{1}{2}\sin(2\theta)(||X||^2-||Y||^2)=0
\end{equation}
that will always be assumed for some $\theta\in\mathbb R$.

\end{remark}

\begin{remark} Let us make one more observation, that is applicable for all convex tube domains $\Omega\subset\mathbb R^n$ containg the origin, namely we have the following property
\begin{equation}
\kappa_{T_{\Omega}}(0;X+iY)=\kappa_{T_{\Omega}}(0;X-iY),\; X+iY\in\mathbb C^n,
\end{equation}
that easily follows from the fact that the competitor in the definition of $\kappa_{T_{\Omega}}(0;\cdot)$ the holomorphic mapping $f:\mathbb D\to T_{\Omega}$ with $f(0)=0$ gives another function $g:\mathbb D\to T_{\mathbb B_n}$ by the formula $g(\lambda):=\overline{f(\overline{\lambda})}$, $\lambda\in\mathbb D$ that satisfies $g(0)=0$, $g^{\prime}(0)=\overline{f^{\prime}(0)}$. 

We even have a stronger regularity under the additional assumption of $\Omega$ being symmetric around $0$: $\kappa_{T_{\Omega}}(0;\pm X\pm iY)=\kappa_{T_{\Omega}}(0;X+iY)=\kappa_{T_{\Omega}}(0;\pm Y\pm iX))$, $X,Y\in\mathbb R^n$.
\end{remark}

Below we assume that $n=2$, which makes no restriction on generality for calculating $\kappa_{T_{\mathbb B_n}}(0;\cdot)$.

\begin{remark} The equation (\ref{eq:orthogonality}) together with the previous symmetries allows us to give a fully effective formula that reduces the problem of calculating the Kobayashi-Royden metric to the special case of a vector $X+iY$, where $X\in[0,\infty)\times\{0\}$ and $Y\in \{0\}\times [0,\infty)$.  
Actually, we easily get the following 
\begin{multline}
||\tilde X||^2=||\cos \theta X-\sin \theta Y||^2=\\
\frac{1}{2}\cos (2\theta)(||X||^2-||Y||^2)+\frac{1}{2}(||X||^2+||Y||^2)-\sin(2\theta)\langle X,Y\rangle.
\end{multline}
The last formula transforms for $\theta$ satisfying (\ref{eq:orthogonality}) to 
\begin{equation}\label{eq:alpha}
\pm\frac{1}{2}\sqrt{4\langle X,Y\rangle^2+(||X||^2-||Y||^2)^2}+\frac{1}{2}(||X||^2+||Y||^2)=:(\alpha(X,Y))^2.
\end{equation}
Analogously we get the equality
\begin{multline}\label{eq:beta}
||\tilde Y||^2=||\cos \theta Y+\sin \theta Y||^2=\\
\mp\frac{1}{2}\sqrt{4\langle X,Y\rangle^2+(||X||^2-||Y||^2)^2}+\frac{1}{2}(||X||^2+||Y||^2)=:(\beta(X,Y))^2.
\end{multline}
Consequently, making additionally use of the formula $\kappa_{T_{\mathbb B_2}}(0;X+iY)=\kappa_{T_{\mathbb B_2}}(0;UX+iUY)$ for arbitrary unitary matrix $U$ we get the following formula that allows us the desired reduction
\begin{equation}
\kappa_{T_{\mathbb B_2}}(0;X+iY)=\kappa_{T_{\mathbb B_2}}(0;(\alpha(X,Y),i\beta(X,Y))).
\end{equation}
\end{remark}


Denote the function defined by (\ref{eq:poisson}) and $a=(\cos\theta,i\sin\theta)$ for some $\theta\in\mathbb R$ by $f_{\theta}$.
Then we get
\begin{multline}
f_{\theta}^{\prime}(0)=
\frac{1}{\pi}\int_0^{2\pi}\frac{\cos\theta(\cos^2t-i\sin t\cos t),\sin\theta(i\sin^2t-\sin t\cos t))dt}{\sqrt{\cos^2\theta \cos^2 t+\sin^2\theta \sin^2 t}}.
\end{multline}
It is elementary to see that the integral $\int_0^{2\pi}\frac{\sin t\cos tdt}{\sqrt{\cos^2\theta\cos^2 t+\sin^2\theta\sin^2 t}}$ vanishes which lets us write
\begin{multline}
f_{\theta}^{\prime}(0)=
\frac{1}{\pi}\int_0^{2\pi}\frac{(\cos\theta\cos^2 t,i\sin\theta\sin^2t)dt}{\sqrt{\cos^2\theta\cos^2 t+\sin^2\theta \sin^2 t}}=\\
\frac{4}{\pi}\int_0^{\frac{\pi}{2}}\frac{(\cos\theta\cos^2 t,i\sin\theta\sin^2t)dt}{\sqrt{\cos^2\theta\cos^2 t+\sin^2\theta \sin^2 t}}.
\end{multline}
After taking into considerations possible reductions, making additionally use of unitary transformation $U$ applied simultanously to both real and imaginary part of elements we see that we may reconstruct all the complex geodesics in $T_{\mathbb B_b}$ passing through the origin by considering the geodesics $f_{\theta}$.

Note that $f_{0}^{\prime}(0)=(4/\pi,0)$, $f_{\pi/2}^{\prime}(0)=(0,i\frac{4}{\pi})$, $f_{\pi/4}^{\prime}(0)=(1,i)$. The symmetry properties of $\kappa_{T_{\mathbb B_2}}(0;\cdot)$ allow us to reduce the problem of caluculating $f_{\theta}^{\prime}(0)$ to the case $\theta\in [0,\pi/2)$. The property $f_{\pi/2-\theta}^{\prime}(0)=(\frac{1}{i}(f_{\theta})_2^{\prime}(0),i(f_{\theta})_1^{\prime}(0))$ lets us reduce the problem to $\theta\in[0,\pi/4]$. We also see that all the directions from the indicatrix $I_{T_{\mathbb B_2}}(0)$ that are lying in $\mathbb R\times i\mathbb R$ are represented by the derivative $f_{\theta}^{\prime}(0)$ for some $\theta$. 

It is elementary to find out that we get the following equalities
\begin{multline}
f_{\theta}^{\prime}(0)=\\
\frac{4}{\pi}\left(\frac{E(\sqrt{1-\tan^2\theta})-\tan^2\theta K(\sqrt{1-\tan^2\theta})}{1-\tan^2\theta},i\tan\theta\left(\frac{K(\sqrt{1-\tan^2\theta})-E(\sqrt{1-\tan^2\theta})}{1-\tan^2\theta}\right)\right),
\end{multline}
where $K$, $E$ are {\it complete elliptic integrals of the first and second type}:
\begin{equation}
K(k):=\int_0^1\frac{\sqrt{1-k^2x^2}dx}{\sqrt{1-x^2}},\; E(k):=\int_0^1\frac{dx}{\sqrt{1-x^2}\sqrt{1-k^2 x^2}}.
\end{equation}

Recall the identities
\begin{equation}
K^{\prime}(k)=\frac{E(k)-(1-k^2)K(k)}{k(1-k^2)},\; K^{\prime}(k)=\frac{E(k)-K(k)}{k}.
\end{equation}
Consequently, we may express the formula as follows:
\begin{equation}\label{eq:f-theta}
f_{\theta}^{\prime}(0)=\frac{2\tan\theta}{\pi\sqrt{1-\tan^2\theta}}\left(K^{\prime}(\sqrt{1-\tan^2\theta})\tan\theta,-iE^{\prime}(\sqrt{1-\tan^2\theta})\right).
\end{equation}
Define now
\begin{align}
\alpha(X,Y)&:=\sqrt{\frac{1}{2}\sqrt{4\langle X,Y\rangle^2+(||X||^2-||Y||^2)^2}+\frac{1}{2}(||X||^2+||Y||^2)},\label{align:alpha}\\
\beta(X,Y)&:=\sqrt{-\frac{1}{2}\sqrt{4\langle X,Y\rangle^2+(||X||^2-||Y||^2)^2}+\frac{1}{2}(||X||^2+||Y||^2)}.\label{align:beta}
\end{align}

In fact the values of $f_{\theta}^{\prime}(0)$ parametrize the bundary of the indicatrix $I_{T_{\mathbb B_2}}(0)$ intersected with $\mathbb R\times i\mathbb R$ by the formula
\begin{equation}
\Psi:\mathbb S^1\owns(\cos\theta,\sin\theta)\to f_{\theta}^{\prime}(0).
\end{equation}

In the further considerations we would be interested in the cases $\theta\in[0,\pi/4]$ so to make the expression well defined we recall that in the limit cases we have $f_0^{\prime}(0)=(4/\pi,0)$, $f_{\pi/4}^{\prime}(0)=(1,i)$.

Denote by 
$\mathfrak K$ the set of all $(x,y)\in\mathbb R^2$ such that $(x,iy)$ belongs to the convex set $I_{T_{\mathbb B^2}}(0)\cap (\mathbb R\times i\mathbb R)$ or equivalently the one bounded by the graph of the curve $\Psi$.


 Then the analysis done above lets us formulate the following formula.

\begin{theorem}\label{thm:formula-kobayashi-ball} Let $X+iY\in\mathbb R^n+i\mathbb R^n\subset\mathbb C^n$ be a non-zero vector. Then the equation
\begin{equation}
rf_{\theta}^{\prime}(0)=(\alpha(X,Y),i\beta(X,Y))
\end{equation}
where the notation is as in (\ref{eq:f-theta}), (\ref{align:alpha}), (\ref{align:beta})
has exactly one solution $(r,\theta)\in(0,\infty)\times [0,\pi/4]$ and we get the equality $\kappa_{\mathbb B_n}(0;X+iY)=r$.
\end{theorem}

\subsection{Schwarz lemma for harmonic mappings with Wirtinger derivatives}
The authors in \cite{For-Kal 2021} asked about the effective formula for the Kobayashi-Royden metric in the tube domain $T_{\mathbb B_2}$ at the origin. Theorem~\ref{thm:formula-kobayashi-ball} gives the answer to that problem in the probably best possible way as it involves the inverses of functions defined by complete elliptic integrals. Forstneric and Kalaj additionally remarked that the formula for the Kobayashi-Royden metric in the domain $T_{\mathbb B_2}$ solves a problem of determining the existence of harmonic self-mappings between the discs in $\mathbb R^2$ with the given derivatives that are problems that were considered in \cite{Bre-Ort-Seip 2021} and \cite{Kov-Yang 2020}. Note also that the formula in Theorem~\ref{thm:formula-kobayashi-ball} may let us produce necessary and sufficient algebraic conditions by a clever inscription of the convex set into $\mathfrak K$ (or the inscription of $\mathfrak K$ into a convex set) and thus may produce estimates similar to the ones given in 
\cite{Bre-Ort-Seip 2021} and \cite{Kov-Yang 2020}. We shall see later how this can be applied in simpler situations. However, it turns out that to provide some Schwarz type lemma for harmonic mappings that could be obtained from Theorem~\ref{thm:formula-kobayashi-ball} it is sometimes more handy to deal with the Wirtinger derivatives.

{\it The Wirtinger operators} are defined as follows
\begin{equation}
\partial:=\frac{1}{2}\left(\frac{\partial}{\partial x}+i\frac{\partial}{\partial y}\right),\;\overline{\partial}:=\frac{1}{2}\left(\frac{\partial}{\partial x}-i\frac{\partial}{\partial y}\right),
\end{equation}
which implies that the {\it the Wirtinger derivatives} for $C^1$-mapping $u=(u_1,u_2):\mathbb D\to\mathbb D$ are given by the following formulae
\begin{equation}
\overline{\partial} u=\frac{1}{2}\left(u_{1x}+u_{2y}+i(u_{2x}-u_{1y})\right),\;\partial u=\frac{1}{2}\left(u_{1x}-u_{2y}+i(u_{2x}+u_{1y})\right).
\end{equation}
It turns out that the functions $\alpha,\beta$ from the previous section when applied to $f^{\prime}(0)$ where $\re f=u$ may be nicely formulated with the help of the Wirtinger derivatives. After elementary calculations we get that when $u:\mathbb D\to\mathbb D$ is harmonic and $f:\mathbb D\to T_{\mathbb B_2}$ is such that $\re f=u$ we get that $f^{\prime}(0)=u_{x}-iu_{y}$. Consequently for $X+iY=f^{\prime}(0)=u_x-iu_y$ after some elementary calculations we have
\begin{equation}
\pm\frac{1}{2}\sqrt{4\langle X,Y\rangle^2+(||X||^2-||Y||^2)^2}+\frac{1}{2}(||X||^2+||Y||^2)=(|\partial u|\pm |\overline{\partial} u|)^2.
\end{equation}
Consequenly, we have the following version of the Schwarz Lemma for harmonic mappings.

\begin{theorem} There is a harmonic mapping $u:\mathbb D\to\mathbb D$ with $u(0)=0$, $a=\partial u(0)$, $b=\overline{\partial}u(0)$ iff $(|a|+|b|,|a|-|b|)\in \mathfrak K$.
\end{theorem}

\begin{remark}
For instance the last formula lets formulate the following implication (Schwarz Lemma for harmonic mappings as in \cite{Kov-Yang 2020}):

If $u:\mathbb D\to\mathbb D$ is a harmonic mapping with $u(0)=0$ then $|\partial u(0)|+|\overline{\partial}u(0)|\leq \frac{4}{\pi}$.

We also get the following standard result.

If $u:\mathbb D\to(-1,1)$ is a harmonic function with $u(0)=0$ then $|\nabla u(0)|\leq \frac{4}{\pi}$.

We may provide the following sufficient condition.

If $\max\{|a|,|b|\}\leq\frac{2}{\pi}$ then there is a harmonic $u:\mathbb D\to\mathbb D$ with $u(0)=0$, $\partial u(0)=a$, $\overline{\partial} u(0)=b$.



In any case to formulate some more effective versions of the infinitesimal Schwarz Lemma for harmonic mappings self-mappings of the unit disc fixing the origin the problem reduces to the study of the geometry of the set $\mathfrak K$ (or the graph of the mapping $\Psi$).
 
\end{remark}

\end{document}